\documentclass{amsart}
\usepackage{amsmath,amsfonts,amssymb,amscd,amsbsy,upref,mathrsfs}
\usepackage{amsfonts}
\usepackage{cleveref}
\usepackage[T1]{fontenc}
\newtheorem{theorem}{Theorem}[section]
\newtheorem{lemma}[theorem]{Lemma}
\newtheorem{proposition}[theorem]{Proposition}
\newtheorem{corollary}[theorem]{Corollary}

\newtheorem{fact}[theorem]{Fact}
\newtheorem{claim}{Claim}
\newtheorem{definition}[theorem]{Definition}

\newtheorem{remark}[theorem]{Remark}

\newcommand{\bl}{\begin{lemma}}
\newcommand{\el}{\end{lemma}}
\newcommand{\bfa}{\begin{fact}}
\newcommand{\efa}{\end{fact}}
\newcommand{\bpr}{\begin{proposition}}
\newcommand{\epr}{\end{proposition}}
\newcommand{\bp}{\begin{proof}}
\newcommand{\ep}{\end{proof}}
\newcommand{\bd}{\begin{definition}}
\newcommand{\ed}{\end{definition}}
\newcommand{\bt}{\begin{theorem}}
\newcommand{\et}{\end{theorem}}
\newcommand{\bc}{\begin{corollary}}
\newcommand{\ec}{\end{corollary}}
\newcommand{\bn}{\begin{notation}}
\newcommand{\en}{\end{notation}}
\newcommand{\br}{\begin{remark}}
\newcommand{\er}{\end{remark}}
\newcommand{\bcl}{\begin{claim}}
\newcommand{\ecl}{\end{claim}}

\newcommand{\N}{{\mathbb{N}}}

\newcommand{\e}{\varepsilon}

\newcommand{\bnum}{\begin{enumerate}}
\newcommand{\enum}{\end{enumerate}}

\numberwithin{subsection}{section}
\numberwithin{equation}{section}

\newcommand{\vvert}[1][\cdot]{\vert #1\vert}
\newcommand{\norm}[1][\cdot]{\lVert #1\rVert}

\DeclareMathOperator{\supp}{supp}

\begin{document}

\title{Operators in tight by support Banach spaces}
\author{Antonis Manoussakis, Anna Pelczar-Barwacz}
\thanks{}
\address[A. Manoussakis]{Department  of Environmental Engineering, Technical University of Crete,\\ GR 73100, Greece}
\email{amanousakis@isc.tuc.gr}
\address[A. Pelczar-Barwacz]{Institute of Mathematics, Faculty of Mathematics and Computer Science, Jagiello\-nian University, {\L}ojasiewicza 6, 30-348 Kra\-k\'ow, Poland}
\email{anna.pelczar@im.uj.edu.pl}
\thanks{Research partially supported by $\mathrm{API\Sigma TEIA}$ program/1082}

\begin{abstract}
We answer the question of W.T.~Gowers, giving an example of a bounded operator on a subspace of Gowers unconditional space which is not a strictly singular perturbation of a restriction of a diagonal operator. We make some observations on operators in arbitrary tight by support Banach space, showing in particular that in such space no two isomorphic infinitely dimensional subspaces form a direct sum. 
\end{abstract}

\maketitle

In \cite{GM0} W.T.~Gowers and B.~Maurey built the first hereditarily indecomposable (HI) Banach space $X_{GM}$, i.e. a space  whose none infinitely dimensional subspace admits a non-trivial bounded projection. They proved also that any operator on a subspace of $X_{GM}$ is a strictly singular perturbation of a multiple of the identity. Recall that an operator is strictly singular if none of its restriction to an infinitely dimensional subspace is an isomorphism onto its image. Gowers-Maurey construction opened the field of study of spaces with a small family of bounded operators. The celebrated space of S.A.~Argyros and R.~Haydon \cite{AH} provided an extreme example in the area; their space is an $\mathscr{L}_\infty$ HI space, on which any bounded operator is a compact perturbation of a multiple of the identity. 

A natural question arises how small family of bounded operators on Banach spaces with an unconditional basis could be. Obviously all diagonal operators with uniformly bounded entries are continuous on such space, therefore the most one can expect is a hereditary "diagonal + strictly singular" property: any bounded operator on a subspace of the space is a strictly singular perturbation of a restriction of a diagonal operator. 

Among the properties to be considered in this context are different types of tightness, studied in \cite{FR1,FR2}, which describe the structure of the family of isomorphisms inside the space. The strongest type is tightness by support. Recall that a  Banach space $X$ with a basis is tight by support if no two disjointly supported infinitely dimensional subspaces of $X$ are isomorphic \cite{FR1}. Any tight by support basis is necessarily unconditional. The typical example of a Banach space tight by support is Gowers unconditional space $X_U$, the unconditional version of Gowers-Maurey space \cite{FR2,G2}. It follows easily that the hereditary "diagonal + strictly singular" property implies tightness by support. W.T.~Gowers asked if the implication can reversed (Problem 5.12 \cite{G2}), in particular if $X_U$ has hereditary "diagonal + strictly singular" property  (Problem 5.13 \cite{G2}). It is known that any bounded operator on the whole space $X_U$ is a strictly singular perturbation of a diagonal operator \cite{GM0}. Adapting arguments of \cite{ABR} one can prove analogous result for any bounded operator $T: Y\to Y$, where $Y$ is a block subspace of $X_U$. W.T.~Gowers \cite{G3} also proved that any isomorphism between block subspaces of a tight by support Banach space is a strictly singular perturbation of a restriction of an invertible diagonal operator.

We answer the questions by constructing a bounded projection on a direct sum of two block subspaces of $X_U$ which is not a strictly singular perturbation of a restriction of a diagonal operator (Theorem \ref{gow2}). The construction uses the block sequence of \cite{KL} in Schlumprecht space generating an $\ell_1$-spreading model and canonical properties of Gowers unconditional space, thus can be easily adapted to other spaces of Gowers-Maurey type and leaves open the question on example of a Banach space with the hereditary "diagonal + strictly singular" property. We reproduce next the construction in arbitrary block subspace of $X_U$ using the results of \cite{FS} (Theorem \ref{gow3}). 

We prove also positive results on bounded operators on arbitrary Banach space $X$ with a tight by support basis. In particular we show that any bounded operator on a subspace generated by a weakly null sequence $(x_n)$ in such space has a restriction to a subspace generated by some subsequence $(x_{k_n})$ of the form $S+D|_{[x_{k_n}]}$, with $S$ strictly singular and $D$ diagonal (Theorem \ref{th1}). If we allow restricting to a block subspace we can replace the diagonal operator $D$ by a multiple of the identity (Theorem \ref{id}), which implies that no two isomorphic infinitely dimensional subspaces of $X$ form a direct sum (Corollary \ref{cor}). 

In case of Gowers unconditional space one can strengthen Theorem \ref{th1} - we prove that any bounded operator on a block subspace $Y$ of $X_U$ into $X_U$ is of the form $S+D|_Y$, with $S$ strictly singular and $D$ diagonal, generalizing earlier results (Proposition \ref{gow1}).

\

We recall briefly the standard notation. Given any $E,F\subset \N$ we write $E<F$, if $\max E<\min F$. Let $X$ be a Banach space with a basis $(e_i)$. Given any $G\subset\N$ by $P_G$ we denote the projection $X\to [e_i: i\in G]$. The \textit{support} of a vector $x=\sum_i x_i e_i$ is the set $\supp x =\{ i\in \N : x_i\neq 0\}$. The support of a subspace $Y$ is the union of supports of all elements of $Y$. We write $x<y$ for vectors $x,y\in X$, if $\supp x<\supp y$. Any sequence $(x_n)\subset X$ with $x_{1}<x_{2}<\dots$ is called a  \textit{block sequence}, a closed subspace spanned by an infinite block sequence $(x_n)$ is called a \textit{block subspace}. Given any basic sequence $(x_n)$ by $[x_n]$ we denote the closed vector space spanned by $(x_n)$.

\section{GENERAL CASE}
In this section we show some positive results on bounded operators on Banach spaces which are tight by support. 
We recall
\begin{definition}\cite{FR1}
 A basis of a Banach space is called tight by support, if no two infinitely dimensional subspaces with disjoint supports are isomorphic.
\end{definition}
\begin{remark}\cite{FR1}\label{unc}
 A tight by support basis is unconditional.
\end{remark}
We are grateful to Valentin Ferenczi for bringing to our attention Remark \ref{unc}.

Throughout this section $X$ denotes a Banach space with a tight by support basis $(e_i)$. The main tool is provided by the following decomposition result, which uses the notion of a diagonal-free operator. We call an operator $R$ defined on a block subspace $[x_n]\subset X$ into $X$ \textit{diagonal-free} provided $\supp x_n\cap\supp Rx_n=\emptyset$ for any $n\in\N$. 
\begin{proposition}\label{prop1} 
Let $X$ be a Banach space with a tight by support basis $(e_i)$. Let $T$ be a bounded operator on a block subspace $[x_n]\subset X$. Then $T=D|_{[x_n]}+S+R$ for some bounded operators $D,S,R$ with $D$ diagonal, $S$ strictly singular and $R$  diagonal-free.

Moreover, if $T$ satisfies $\supp Tx_n\cap \supp x_m=\emptyset$ for any $n\neq m$, then the above formula holds with $R=0$.
\end{proposition}

\begin{proof} Let $(x_n)$ be a normalized block basis and $T:[x_n]\to X$ be a bounded operator  with $\norm[T]=C>0$.  Since $X$ is tight by support, the operator $P\circ T$, where $P$ is the projection on $[e_i:\ i\not\in \cup\supp(x_{n})]$ is strictly singular. Thus we can assume that $\cup_{n}\supp(x_n)=\N$.

For any $n,k\in\N$ put 
$$
A_{n,k}=\{i\in\supp x_n: \ |x_n(i)|\leq \frac{1}{2^k}|Tx_n(i)|\}
$$ 
and $A_k=\cup_{n\in\N}A_{n,k}$, $k\in\N$. 

For any $k\in\N$ put $T_k=P_k\circ T$, where $P_k$ is the projection from $X$ onto $[e_i: \ i\in A_k]$, and let $D_k: X\to X$ be the diagonal operator defined by $D_{k}(e_{i})=\lambda_{i}e_{i}$ where 
$$
\lambda_{i}=\begin{cases}0\,\,\, &\textrm{if}\,\,\ i\in A_k,\\
Tx_n(i)/x_n(i) &\textrm{if}\,\, i\in\supp x_n\setminus A_k=\supp
x_n\setminus A_{n,k}.
\end{cases}
$$ 
By the definition of $A_{n,k}$'s we have $\norm[D_k]\leq 2^k$. 

Fix $k\in\N$ and assume that $T_k$ is not strictly singular. Thus $T_k$ is an isomorphism between some infinitely dimensional subspaces $U\subset [x_n]$ and $W\subset P_k(X)$. Consider a bounded operator $R_k=(Id-P_k)\circ (T_k|_U)^{-1}: W\to [e_i: \ i\not\in A_k]$.

As $X$ is tight by support and $\supp R_k(W)\cap \supp W=\emptyset$, there is some infinitely dimensional subspace $V\subset W$ such that $\norm[R_k|_V]\leq (2C)^{-1}$. As $(T_k|_U)^{-1}$ is an isomorphism, the subspace $Z=(T_k)^{-1}(V)$ is also infinitely dimensional. Take $x\in (T_k)^{-1}(V)$ and compute 
\begin{align*}
\norm[T_kx] &\leq \norm[Tx]\leq C\norm[x]
\\
&\leq C\norm[P_kx]+C\norm[x-P_kx]\notag
\\ 
&=C\norm[P_kx]+C\norm[R_k(T_kx)]
\notag 
\\
&\leq C\norm[P_kx]+\frac{1}{2}\norm[T_kx].
\notag
\end{align*}
Hence $\norm[T_kx]\leq 2C\norm[P_kx]$ for any $x\in Z$. As $Z\subset U$ also $T_k|_Z$ is an isomorphism onto its image.

On the other hand for any $x\in [x_n]$ and $i\in \supp P_kx\subset A_k$ we have 
$$
|P_kx(i)|\leq \frac{1}{2^k}|Tx(i)|=\frac{1}{2^k}|T_kx(i)|.
$$ 
It follows that for any $x\in [x_n]$ we have $\norm[P_kx]\leq \frac{1}{2^k}\norm[T_kx]$, which for sufficiently big $k$ gives contradiction for any non-zero $x\in Z$. Therefore for sufficiently big $k$ the operator $T_k$ is strictly singular. 

Now we have 
\begin{align*}
&(D_k|_{[x_n]}+ T_k-T)(\sum_na_nx_n)\\
&=\sum_na_n\sum_{i\in\supp x_n}\lambda_ix_n(i)e_i+\sum_na_nP_kTx_n-\sum_na_nTx_n\\
&=\sum_na_n\sum_{i\in\supp x_n\setminus A_{n,k}}Tx_n(i)e_i+\sum_na_n\sum_{i\in A_k}Tx_n(i)e_i+\sum_na_n\sum_{i\in\N}Tx_n(i)e_i\\
&=\sum_na_n\sum_{i\in\N\setminus (A_k\cup\supp x_n)}Tx_n(i)e_{i}\,.
\end{align*}
Therefore the operator $R=T-D_k|_{[x_n]}-T_k$ is diagonal-free.

Now if we assume that $\supp Tx_m\cap \supp x_n=\emptyset$ for $n\neq m$, then, as we assumed that $\supp[x_n]=\N$, we have that $\supp Tx_n\subset \supp x_n$ for any $n\in\N$.

Then $T=D_k|_{[x_n]}+T_k$, as
\begin{align*}
\left(D_k|_{[x_n]}+T_k\right)(\sum_na_nx_n)&=\sum_na_n\sum_{i\in\supp x_n}\lambda_ix_n(i)e_i+\sum_na_nP_kTx_n\\
&=\sum_na_n\sum_{i\in\supp x_n\setminus A_{n,k}}Tx_n(i)e_i+\sum_na_n\sum_{i\in A_k}Tx_n(i)e_i\\
&=\sum_na_nTx_n\,.
\end{align*}
For the last equality recall that $\supp Tx_n\subset \supp x_n$ for any $n\in\N$.
\end{proof}
Proposition \ref{prop1} implies immediately the following result.
\begin{theorem}\label{th1} 
Let $X$ be a Banach space with a tight by support basis.

Let $T:[x_n]\to X$ be a bounded operator on a subspace spanned by a weakly null sequence $(x_n)\subset X$. Then there exists a subsequence $(x_{n})_{n\in M}$ such that $T|_{[x_{n}:\ n\in M]}=D|_{[x_{n}:\ n\in M]}+S$, where $ D:X\to X$  is a bounded diagonal operator  and   $S: [x_n: \ n \in M]\to X$  is a bounded strictly singular operator.

In particular the assertion holds if $(x_n)$ is a block sequence.
\end{theorem}
We can replace diagonal operator by a multiple of the identity, if we allow passing to a block sequence instead of subsequence.
\begin{theorem}\label{id}
Let $X$ be a Banach space with a  tight by support basis. 

Let $T:[x_n]\to X$ be a bounded operator on a block subspace $[x_n]\subset X$. Then there is an infinitely dimensional block subspace $W\subset [x_n]$ such that $T|_W=\alpha Id|_W+S$, for some scalar $\alpha$ and bounded strictly singular operator $S:[x_n]\to X$.
\end{theorem}
As for any isomorphism $T$ any scalar $\alpha$ given by the above Theorem is non-zero, we obtain the following 
\begin{corollary}\label{cor}
Let $X$ be a Banach space with a tight by support basis. 

Then for any isomorphic infinitely dimensional subspaces $Y,Z\subset X$  we have $\inf\{\norm[y-z]: \ y\in Y, \ z\in Z, \ \norm[y]=\norm[z]=1 \}=0$. 
\end{corollary}
\begin{proof}[Proof of Theorem \ref{id}] By Remark \ref{unc} we can assume that the basis of $X$ is 1-unconditional and the sequence $(x_n)$ is normalized. Passing to a further subspace by Theorem \ref{th1} we can assume that $T|_{[x_n]}=D|_{[x_n]}+S$ with $D$ bounded diagonal with entries $(\lambda_n)$ and $S$ compact.  Let $\Lambda=\sup_n|\lambda_n|$ and assume $\Lambda >0$. 

We shall prove the following
\begin{claim}\label{claim1}
For any $\e>0$ in any block subspace of $[x_n]$ there is a further block subspace $[y_m]$ and some $\alpha_\e$ with $|\alpha_\e|\leq\Lambda$ such that 
$$
\norm[(D-\alpha_\e Id)|_{[y_m]}]<\e.
$$
\end{claim}
Having Claim \ref{claim1} consider a cluster point $\alpha_0$ of $(\alpha_\e)_{\e>0}$ and pick some sequence $(\alpha_n)$ and descending sequence of block subspaces $Y_n$ such that $\vvert[\alpha_n-\alpha_0]<1/2^n$ and $\norm[(D-\alpha_n Id)|_{Y_n}]<1/2^n$. Thus $\norm[(D-\alpha_0 Id)|_{Y_n}]<1/2^{n-1}$ and on the diagonal subspace $Y_0$ of $(Y_n)$ the operator $(D-\alpha_0 Id)|_{Y_0}$ is compact which will finish the proof.

Proof of Claim \ref{claim1}. Fix $\e>0$  and  consider a partition  of $\{\lambda: \ |\lambda|\leq \Lambda\}=\cup_{i=1}^{d}A_i$ into pairwise disjoint subsets of diameter smaller than $\e/2$. For every $n$ put $I_{n,i}=\{k\in\supp x_n: \lambda_{k}\in A_i\}$ and $x_{n,i}=x_{n}{|}_{I_{n,i}}$. By the unconditionality  we get $\norm[x_{n,i}]\leq 1$. As $X$ is tight by support, for every $i\ne j$ any restriction to a linear subspace spanned by a block sequence of $(x_{n,i})_n$ of the operator
$$
M_{i,j}:lin\{x_{n,i}: n\in\N\}\ni\sum_na_nx_{n,i}\to \sum_na_nx_{n,j}\in lin\{x_{n,j}: n\in\N\}
$$
is either non-bounded or strictly singular. Using this observation in any block subspace of $(x_n)$ we can find a further block sequence $(y_m)$ satisfying for some $i_{0}\leq d$ the following
$$
\norm[y_m|_{\cup_nI_{n,i_{0}}}]=1, m\in\N\,\,\, \text{ and} \,\,\,\norm[y_m|_{\cup_nI_{n,i}}]\to 0, m\to\infty \,\,\,\text{ for } i\neq i_0\,.
$$
The above statement can be easily proved by induction on $d$. Passing to a subsequence of $(y_m)$ we can assume that $\norm[P_{\N\setminus\cup_nI_{n,i_0}}|_{[y_m]}]<\e/(4\Lambda)$.

Pick any scalar $\lambda_\e\in A_{i_0}$ and compute for any vector $\sum_nb_nx_{n}\in [y_m]$ of norm 1:
\begin{align*}
&\norm[\lambda_\e\sum_{n}b_nx_{n}-D(\sum_{n}b_nx_{n})]\leq \\
&\leq\norm[\sum_nb_n\sum_{k\in I_{n,i_0}}(\lambda_\e-\lambda_k)x_n(k)e_k]+\norm[\sum_nb_n\sum_{k\not\in I_{n,i_0}}(\lambda_\e-\lambda_k)x_n(k)e_k]\\
&\leq \max_{k\in \cup_n I_{n,i_0}}|\lambda_k-\lambda_\e|\norm[\sum_nb_nx_n]+\max_k|\lambda_k-\lambda_\e|\norm[\sum_n\sum_{i\neq i_0}b_nx_{n,i}]\\
&\leq \frac{\e}{2}\norm[\sum_nb_nx_n]+2\Lambda\norm[P_{\N\setminus\cup_nI_{n,i_0}}|_{[y_m]}]\leq \e
\end{align*}
which proves that $\norm[(D-\lambda_\e Id)|_{[y_m]}]\leq \e$.

\end{proof}

Let $X$ be a Banach space with an unconditional basis $(e_i)$. In the next section we shall use the following general observation concerning the form of a projection on one of the component of a direct sum formed by two block subspaces with possibly coinciding supports. Assume we have block subspaces $Y=[y_n]$ and $Z=[z_n]$ with 
\begin{enumerate}
\item[(D1)] $\min\{\supp y_{n+1}, \supp z_{n+1}\}\geq \max\{\supp y_n,\supp z_n\}, \ \  n\in\N$
\item[(D2)] $\inf\{\norm[y-z]: \norm[y]=\norm[z]=1, y\in Y, z\in Z\}>0$ 
\end{enumerate}
Consider projections $P_Y: Y+Z\ni y+z\mapsto y\in Y$, $P_Z: Y+Z\ni y+z\mapsto z\in Z$. By (D2) these projections are bounded. 

\begin{lemma}\label{fact1}
In the situation as above the projection $P_Y$ is of the form $P_Y=D|_{Y+Z}+S$, with $S$ strictly singular and $D:X\to X$ diagonal if and only if there is a partition $\N=F\cup G$ such that $P_G|_Y$ and $P_F|_Z$ are strictly singular. Moreover, if either of the conditions hold, the diagonal operator $D$ can be chosen to be a projection onto a subspace spanned by a subsequence of the basis.
\end{lemma}

\begin{proof}
Assume that $P_Y$ is of the form $P_Y=D|_{Y+Z}+S$, with $S$ strictly singular and $D:X\to X$ diagonal with entries $(\lambda_i)$. Let 
$$
F=\left\{i\in\N: \ |\lambda_i|>\frac{1}{2}\right\}, \ \ \ G=\left\{i\in\N: \ |\lambda_i|\leq \frac{1}{2}\right\}
$$ 
Then for any $y=\sum_na_ny_n\in Y$ we have $y=P_Yy=Dy+Sy$, so
$$
\sum_na_n\sum_{i\in\supp y_n} y_n(i)e_i=\sum_na_n\sum_{i\in\supp y_n}\lambda_i y_n(i)e_i + S(\sum_na_ny_n)
$$
Thus
$$
\sum_na_n\sum_{i\in\supp y_n}(1-\lambda_i)y_n(i) e_i=S(\sum_na_ny_n)
$$
Applying the projection $P_G$ we get
$$
\sum_na_n\sum_{i\in\supp y_n\cap G}(1-\lambda_i)y_n(i) e_i=(P_G\circ S)(\sum_na_ny_n)
$$
thus by unconditionality of $(e_i)$
\begin{align*}
\norm[(P_G\circ S)(\sum_na_ny_n)]&=\norm[\sum_na_n\sum_{i\in\supp y_n\cap G}(1-\lambda_i)y_n(i) e_i]\\
&\geq \frac{1}{2}\norm[\sum_na_n\sum_{i\in\supp y_n\cap G}y_n(i)e_i]\\
&=\frac{1}{2}\norm[\sum_na_nP_Gy_n]\\ 
&=\frac{1}{2}\norm[P_G(\sum_na_ny_n)].
\end{align*}
As $S$ is strictly singular also $P_G|_Y$ is strictly singular. Analogously we prove that $P_F|_Z$ is strictly singular.

The reverse implication follows straightforward. Given suitable $F,G$ we write $P_Y=P_F|_{Y+Z}+P_G\circ P_Y-P_F\circ P_Z$. By the assumption on projections $P_F$, $P_G$ on corresponding subspaces the operator $P_G\circ P_Y-P_F\circ P_Z$ is strictly singular. This reasoning proves also the "moreover" part of the lemma. 
\end{proof}

\section{GOWERS UNCONDITIONAL SPACE CASE}
In this section we answer Gowers' question \cite{G2} by giving an example of an operator $T$ on a subspace $W$ of Gowers unconditional space $X_U$ which is not of the form $D|_W+S$ with $D$ diagonal and $S$ strictly singular. We present first the list of canonical properties of the class of spaces of Gowers-Maurey type that are needed for our construction and proceed to the proof of the main result.  Next we generalize the construction to any block subspace of $X_U$ proving that an operator which is not a strictly singular perturbation of a restriction of a diagonal operator can be built inside any infinitely dimensional subspace of $X_U$. However, performing the construction inside block subspaces requires more technical background concerning spaces of Gowers-Maurey type, thus we state it separately for the convenience of the reader. We close the section with proving that even though the "diagonal + strictly singular" property does not hold for any infinitely dimensional subspace of the space $X_U$, it is satisfied for block subspaces of $X_U$. 

\

We recall now the definition of Schlumprecht space $S$ and Gowers unconditional space $X_U$. The spaces are defined as a completion of $c_{00}$ under suitable norm, defined as a limit of an increasing sequence of norms. 

Let $f$ denote the function $x\mapsto \log_2(x+1)$. The norm $\norm_S$ of Schlumprecht space $S$ satisfies on $c_{00}$ the following equation.
$$  
\norm[x]_S=\max\left\{\norm[x]_{\infty},\sup_{n}\frac{1}{f(n)}\sup\{\sum_{i=1}^{n}\norm[E_{i}x]_S: E_{1}<\dots<E_{n}\}\right\}.
$$
It follows straightforward that the basis $(\tilde{e}_n)$ of $S$ is 1-unconditional and subsymmetric, i.e. equivalent to any of its infinite subsequences. 

We shall sketch the definition of Gowers unconditional space $X_U$, referring to \cite{G} for details, and state properties of the space we need in a list of facts given below. 

The norm of $X_U$ satisfies on $c_{00}$ the following implicit equation
\begin{align*}
\norm[x]=\max\{\norm[x]_{\infty},&\sup_{n}\frac{1}{f(n)}\sup\{\sum_{i=1}^{n}\norm[E_{i}x]:  E_{1}<\dots<E_{n}\}, \\
  &\sup\vvert[x^*(x)]: x^* \text{ special functional of length }k, k\in K\}
\end{align*}
for some fixed infinite and co-infinite $K\subset\N$. Special functionals are described with the use of so-called coding function $\sigma$ defined on the family $\mathbf{Q}$ of finite sequences of vectors with rational coordinates with modulus at most 1, taking values in $\N\setminus K$ and satisfying certain growth condition. A special functional of length $k$ is of the form $\frac{1}{\sqrt{f(k)}}\sum_{j=1}^kx_j^*$, for some block sequence $(x_1^*,\dots,x_k^*)$ with each $x_j^*$ of the form $x_j^*=\frac{1}{f(n_j)}\sum_{n=1}^{n_j}x_{j,n}^*$, where $(x_{j,1}^*, \dots, x_{j,n_j}^*)$ is a block sequence in $\textbf{Q}$ of vectors with norm at most 1, and  $n_{j+1}=\sigma(|x_1^*|, \dots, |x_j^*|)$ for any $j=2,\dots,k$. 

Recall that in case of Gowers-Maurey space the coding function depends on $(x_1^*, \dots, x_j^*)$, not on $(|x_1^*|, \dots, |x_j^*|)$, which  makes the space hereditarily indecomposable. In case of Gowers space the basis is 1-unconditional, but including special functionals in the norming set forces tightness by support. 

\

The basic tools are formed by sequences of $\ell_1$-averages. A vector $x\in X_U$ is called an $\ell_{1}^{n}$-\textit{average with constant} $c\geq 1$, $n\in\N$, if  $x=\frac{1}{n}(x_{1}+\dots+x_{n})$ for some block sequence $x_1<\dots<x_n$ with $\norm[x_{i}]\leq  1$ and  $\norm[x]\geq 1/c$. A block sequence of 
$\ell_1^{n_k}$-averages $(x_k)_{k=1}^N\subset X_U$ is a \textit{rapidly increasing sequence (RIS) of $\ell_1$-averages}, if - roughly speaking - $(n_k)$ increases fast enough, with the length $n_k$ of average $x_k$ depending not only on $k$, but also on the support of $x_{k-1}$, and the length $N$ of the sequence is small with respect to the length $n_1$ of the first average, with all relations described in terms of the function $f$. 

\

We list now the properties of the space $X_U$ needed in the sequel. This list indicates that the results of this section can be easily adapted to the case of other spaces of Gowers-Maurey type. 

First we recall the standard
\begin{fact}[(Lemma 1 \cite{G})]\label{fact4}
For any $n\in\N$ and $c>1$ every block subspace of $X_U$ contains an $\ell_1^n$-average with constant $c$.
\end{fact}
We shall nees also the following simple observation.
\begin{fact}\label{fact5}
 For any sequence $(z_n)$ of $\ell_1$-averages of increasing length and a common constant and any sets $(D_n)$ in $\N$ with $\inf_n\norm[P_{D_n}z_n]>0$ also $(P_{D_n}z_n)$ is a sequence of $\ell_1$-averages of increasing length and a common constant. 
\end{fact}

We state now the canonical property of the space $X_U$, whose variations in different spaces of Gowers-Maurey type or Argyros-Deliyanni type are responsible for the irregular properties of the spaces, such as having a small (in different meanings) family of bounded operators. 

\begin{fact}\label{fact2} 
(a) Fix a seminormalized block sequence $(u_n)\subset X_U$ and a seminormalized block sequence $(v_n)\subset X_U$ of $\ell_1^n$-averages with a constant $c$, satisfying $\supp v_n\cap \supp u_m=\emptyset$, for all $n,m\in\N$. Then there are sequences $(w_k)\subset [u_n]$, $(z_k)\subset [v_n]$ of the form $w_k=\sum_{n\in J_k}a_nu_n$, $z_k=\sum_{n\in J_k}a_nv_n$, such that $\norm[w_k]=1$, $k\in\N$, and $z_k\to 0$, as $k\to\infty$. 

(b) Fix a subsequence $(e_{i_n})$ of the basis of $X_U$. Then there is a normalized sequence $(w_k)\subset [e_{i_n}]$, $w_k=\sum_{n\in J_k}a_ne_{i_n}$, such that $\sup_{(j_n)}\norm[\sum_{n\in J_k}a_ne_{j_n}]\to 0$, as $k\to\infty$, where the supremum is taken over all sequences $(j_n)\subset \N$ with $\{j_n: n\in \N\}\cap \{i_n: n\in \N\}=\emptyset$. 
\end{fact}
The proof of (a) follows directly the lines of the proof in \cite{G} of the fact that the space $X_U$ satisfies assumptions of Lemma 9 \cite{G}. First we pass to an infinite set $N\subset \N$ such that any finite subsequence $(v_{k_1}, \dots,v_{k_N})$ of $(v_n)_{n\in N}$ with $k_1>N$  forms a RIS of $\ell_1$-averages. Now it suffices to take  for any $k\in\N$ a special functional of length $k$ of the form 
$$
\frac{1}{\sqrt{f(k)}}\sum_{j=1}^k\frac{1}{f(n_j)}\sum_{n\in C_j}u_n^*
$$
where $u_n^*(u_n)=1$, $\# C_j=n_j$, $\min C_j>n_j$ and the corresponding vector 
$$
w_k=\frac{\sqrt{f(k)}}{k}\sum_{j=1}^k\frac{f(n_j)}{n_j}\sum_{n\in C_j}u_n.
$$
Then for any $(v_n)$ as above we have 
$$
\norm[\frac{\sqrt{f(k)}}{k}\sum_{j=1}^k\frac{f(n_j)}{n_j}\sum_{n\in C_j}v_n]\leq\e(k,c),
$$ 
for some $\e(k,c)\to 0$ as $k\to\infty$. 

In case of subsequences of the basis the proof is even simplified. 

\

Recall that Facts \ref{fact1} and \ref{fact2} imply immediately the following
\begin{theorem}\cite{FR1,G2}
 The unit vector basis of Gowers unconditional space $X_U$ is tight by support. 
\end{theorem}

The next fact allows for transfer of an example of a sequence needed in Theorem \ref{gow2} from Schlumprecht space to Gowers unconditional space. Recall that a~basic sequence $(x_n)$ generates some subsymmetric basic sequence $(\tilde{x}_n)$ as a~\textit{spreading model}, if for any $(a_i)_{i=1}^k$, $k\in\N$, we have
$$
\lim_{n_1\to\infty}\lim_{n_2\to\infty}\dots\lim_{n_k\to\infty}\norm[\sum_{i=1}^ka_ix_{n_i}]=\norm[\sum_{i=1}^ka_i\tilde{x}_i]\,.
$$
We say that a basic sequence generates an $\ell_1$-spreading model, if it generates the unit vector basis of $\ell_1$ as a spreading model.
\begin{fact}\label{fact3}
 The basis of $X_U$ generates the basis of Schlumprecht space as a spreading model. 
\end{fact}
The proof of this fact follows the lines of the proof of Prop. 3.1. of \cite{AS} of the above result in case of Gowers-Maurey space.

\

Now we are ready to prove the main result.
\begin{theorem}\label{gow2}
 There are block subspaces $Y=[y_n]$, $Z=[z_n]$ in Gowers unconditional space $X_U$ satisfying (D1), (D2) and 
\begin{enumerate}
 \item[(D3)] for any partition $F\cup G=\N$ with $P_F|_Z$ strictly singular the operator $P_G|_Y$ is not strictly singular.
\end{enumerate}
\end{theorem}
By Lemma \ref{fact1} we obtain the following answer to Gowers' Problem 5.13 (and thus also Problem 5.12) \cite{G2}.
\begin{corollary}
There is a bounded operator on a subspace of Gowers unconditional space $X_U$ which is not a strictly singular perturbation of a restriction of a diagonal operator on $X_U$.
\end{corollary}
\begin{proof}[Proof of Theorem \ref{gow2}]
We shall use the seminormalized block sequence of \cite{KL} generating an $\ell_1$-spreading model in Schlumprecht space. Recall that two vectors $u,v$ have the same distribution, if for some increasing bijection $\rho: \supp u\to\supp v$ we have $v(\rho (i))=u(i)$ for each $i\in\supp u$. Let $u_j=\frac{f(j)}{j}\sum_{i=1}^j\tilde{e}_i$. Take $(\tilde{y}_n)\subset S$ to be the block sequence of Theorem 6 of \cite{KL}, i.e. $\tilde{y}_n=\sum_{j=1}^n\tilde{v}_{n,j}$, where $(\tilde{v}_{n,j})_{j=1}^n$ have pairwise disjoint supports carefully designed and each $\tilde{v}_{n,j}$ has the same distribution as $u_{p_j}/2$, for some fixed $p_j\nearrow \infty$.  The sequence $(\tilde{y}_n)$ generates an $\ell_1$-spreading model, as $\norm[\tilde{v}_{n_1,j}+\dots+\tilde{v}_{n_p,j}]\approx p/2$ for $j\gg p$ (cf. \cite{KL}). 

Write $\tilde{y}_n=\sum_{j=1}^n\frac{f(j)}{2j}\sum_{i\in I_j}\tilde{e}_i$, $\# I_j=j$, for each $n$ and consider sequence $(y_n)\subset X_U$ defined as
$$
y_n=\sum_{j=1}^nv_{n,j}=\sum_{j=1}^n\frac{f(j)}{2j}\sum_{i\in K_j}e_i
$$
where $K_j$'s with $\# K_j=j$ are pushed forward along the basis $(e_i)$ so that by Fact \ref{fact3} the vectors $(y_n)$ form a seminormalized block sequence  with the property $\norm[v_{n_1,j}+\dots+v_{n_p,j}]\approx p/2$ for $j\gg p$, therefore also generating an $\ell_1$-spreading model. 

We define the sequence $(z_n)$ in the following way. Take a mapping $\tau:\cup_n\supp y_n\to\N$ such that 
\begin{enumerate}
 \item[(j1)] $\tau|_{\supp y_n}:\supp y_n\to \{1,2,\dots, \#\supp y_n\}$ is a bijection for any $n\in\N$,
\item[(j2)] $\tau(r)\geq \tau(s)$ iff $y_n(r)\leq y_n(s)$ for any $n\in\N$ and $r,s\in\supp y_n$.
\end{enumerate}
Notice that (j1) and (j2) implies the following property.
\begin{enumerate}
 \item[(j3)] $\tau(\supp v_{n,j})=\tau(\supp v_{m,j})$ for any $j\leq n<m$.
\end{enumerate}
Let 
$$
z_n(i)=\frac{1}{4^{\tau(i)}}y_n(i) \ \text{ for any }\  i\in\supp y_n \ \text{ and }\ z_n(i)=0 \ \text{ otherwise.}
$$
In this way we obtain two seminormalized block sequences $(y_n)$ and $(z_n)$ with $\supp y_n=\supp z_n$, thus in particular satisfying (D1). 
 
Roughly speaking, the proof of Theorem \ref{gow2} relies on the following three properties of the above sequences: for any $(i_n)\subset\N$ with $i_n\in\supp y_n$, $n\in\N$, the projection $P_{\{i_n:\ n\in\N\}}|_Y$ is strictly singular (Claim \ref{claim2}), whereas $P_{\{i_n:\ n\in\N\}}|_Z$ is not strictly singular provided $\inf_n z_n(i_n)>0$ (Claim \ref{claim3}). Moreover, the projection on the set containing supports of almost all  $(v_{n,j})_n$ for any $j$ restricted to $Y$ is not strictly singular (Claim \ref{claim4}). 

\

Proof of (D2). Assume towards contradiction that $\inf\{\norm[y-z]: \norm[y]=\norm[z]=1, y\in Y, z\in Z\}=0$. Thus there are some normalized block sequences $(w_k)\subset [y_n]$  and $(v_k)\subset [z_n]$ with $\norm[w_k-v_k]\leq 1/ 16^k$, $k\in\N$.  Thus for any $(c_k)\subset [-1,1]$ we have $\norm[\sum_kc_kw_k-\sum_kc_kv_k]\leq 1/8$. 

Take $(c_k)\subset [-1,1]$, let $w=\sum_kc_kw_k$, $v=\sum_kc_kv_k$ and $I=\{i\in \supp w: \ |w(i)|\geq 2|v(i)|\}$ and compute, using 1-unconditionality of the basis of $X$
\begin{align*}
 \frac{1}{8}\geq \norm[\sum_kc_kw_k-\sum_kc_kv_k]\geq\norm[\sum_{i\in I}(w(i)-v(i))e_i]\geq \frac{1}{2}\norm[\sum_{i\in I}w(i)e_i].
\end{align*}
Analogously compute for $J=\{i\in \supp w: \ |v(i)|\geq 2|w(i)|\}$.

Thus for any $w=\sum_kc_kw_k$ with norm 1 and $v=\sum_kc_kv_k$ we have 
$$
\norm[\sum_{i\in \supp w: \frac{1}{2}|v(i)|<|w(i)|<2|v(i)|}w(i)e_i]\geq 1/4.
$$
Let 
\begin{align*}
w_k&=\sum_{n\in I_k}a_ny_n=\sum_{n\in I_k}a_n\sum_{i\in\supp y_n}y_n(i)e_i, \\
v_k&=\sum_{n\in I_k}d_nz_n=\sum_{n\in I_k}d_n\sum_{i\in\supp y_n}4^{-\tau(i)}y_n(i)e_i.
\end{align*}
For any $i\in \supp w$ we have 
$$
\frac{1}{2}|v(i)|<|w(i)|<2|v(i)|\quad\textrm{iff}\quad \frac{1}{2}|d_n|<4^{\tau(i)}|a_n|<2|d_n|\quad \textrm{where } i\in\supp y_n.
$$ 
Given $n\in\N$ there is at most one $i\in\supp y_n$ satisfying $\frac{1}{2}|d_n|<4^{\tau(i)}|a_n|<2|d_n|$, name it by $i_n$. Hence
$$
1/4\leq\norm[\sum_{i\in \supp u: \frac{1}{2}|v(i)|<|w(i)|<2|v(i)|}w(i)e_i]=\norm[\sum_kc_k\sum_{n\in I_k}a_ny_n(i_n)e_{i_n}]
$$
which implies that for any $(c_k)$ we have
\begin{align*}
\norm[\sum_kc_k\sum_{n\in I_k}a_ny_n]\leq 4\norm[\sum_kc_k\sum_{n\in I_k}a_ny_n(i_n)e_{i_n}] 
\end{align*}
i.e. $(\sum_{n\in I_k}a_ny_n)_k$ and $(\sum_{n\in I_k}a_ny_n(i_n)e_{i_n})_k$ are equivalent.

On the other hand we have the following claim, which yields contradiction. Whereas the above reasoning holds for any $(y_n)$ and $(z_n)$ related by means of a suitable function $\tau$,  the claim uses only the fact that the spreading models of the basis of $X_U$ and of the chosen sequence $(y_n)$ are quite different and the basis of a variant of Schlumprecht space dominates the basis of $X_U$.
\begin{claim}\label{claim2}
The mapping $(y_n)_{n}\to (y_n(i_n)e_{i_n})_{n}$ extends to a strictly singular operator.
\end{claim}
Proof of Claim \ref{claim2}. We shall prove that the mapping carrying $(y_n)_n$ to the standard basis $(\bar{e}_n)$ of some variant of Schlumprecht space - defined by the function $f(x)=\sqrt{\log_2(x+1)}$ instead of $f(x)=\log_2(x+1)$ - is strictly singular. As the basis of such variant of Schlumprecht space is subsymmetric and dominates the basis of Gowers space, thus also the mapping $(y_n)_n\to(y_n(i_n)e_{i_n})_n$ would be strictly singular.
 
We apply results of \cite{S} taking into account that the basis $(\bar{e}_n)$ of a variant of Schlumprecht space is subsymmetric. By Prop. 2.5 \cite{S} the basis $(\bar{e}_n)$ is strongly dominated by $\ell_1$ (according to the Def. 2.1 of \cite{S}) and by  Lemma 2.4 \cite{S} satisfies for some $\delta_k\searrow 0$ and any scalars $(\alpha_n)$ the following
$$
\norm[\sum_n\alpha_n\bar{e}_n]\leq \max_k\delta_k\max_{k\leq n_1<n_2<\dots<n_k}\sum_{i=1}^k|\alpha_{n_i}|.
$$
Now in order to show that the mapping $M: (y_n)_n\to (\bar{e}_n)_n$ is strictly singular repeat part of the proof of Theorem 1.1 of \cite{S}. Take any normalized block sequence $(u_m)$ of $(y_n)$, $u_m=\sum_{i\in J_m}\alpha_iy_i$, $m\in\N$. Passing to a further block sequence, as $X_U$ does not contain $c_0$, we can assume that $\max_{i\in J_m}|\alpha_i|\to 0$ as $m\to\infty$. Given $k_0\in\N$ estimate the norm of $v_m=M(u_m)$ using that $(y_n)$ is unconditional and generates an $\ell_1$-spreading model:
\begin{align*}
\norm[\sum_{i\in J_m}\alpha_i\bar{e}_i]&\leq \max\left\{\max_{k=1,\dots, k_0-1}\delta_1\sum_{k\leq n_1<\dots<n_k, n_i\in J_m}|\alpha_{n_i}|\right.,\left. \max_{k\geq k_0}\delta_{k_0}\sum_{k\leq n_1<\dots<n_k, n_i\in J_m}|\alpha_{n_i}|\right\} \\
& \leq \max\{\delta_1k_0\max_{i\in J_m}|\alpha_i|, 2\delta_{k_0}\norm[u_m]\}\\
& \leq \max\{\delta_1k_0\max_{i\in J_m}|\alpha_i|, 2\delta_{k_0}\}.
\end{align*}
As $\delta_k\to 0$, choosing sufficiently big $k_0$ and $m$ we can force the norm of $v_m$ to be as small as needed, which proves that $v_m\to 0$ and finishes the proof of strict singularity of $M$ and of Claim \ref{claim2}.

Proof of (D3). We introduce some notation first. Given $n\in\N$ and $t\in\tau (\supp y_n)$ let $i_{n,t}\in\supp y_n$ be the unique index $i\in \supp y_n$ with $\tau(i)=t$. Notice that (j3) by definition of $(v_{n,j})_{n,j}$ implies the following 
\begin{enumerate}
\item[(j4)] for any $i\in\supp y_n$, $k\in\supp y_m$ with  $\tau(i)=\tau(k)$ we have $y_n(i)=y_m(k)$.
\end{enumerate}
Thus we can write $y_n=\sum_{t\in \tau(\supp y_n)}\gamma_te_{i_{n,t}}$ for some  scalars $(\gamma_t)_{t\in\N}\subset [0,1]$. 
Given any $t\in\N$ let also $N_t=\{n\in\N:\ t\in\tau (\supp y_n)\}$. 

The property (D3) follows from the next two claims. The first one is based only on properties of the subsequences of the basis described in the Fact \ref{fact2}(b). 
\begin{claim}\label{claim3}
Take $F\subset\N$ with $P_F|_Z$ strictly singular. Then for any $t\in\N$ the set $\{i_{n,t}: \ n\in N_t\}\cap F$ is finite.
\end{claim}
Proof of Claim \ref{claim3}. Assume that for some $t_0\in\N$ the set $H=\{i_{n,t_0}: n\in N_{t_0}\}\cap F$ is infinite. We shall prove that the projection $P_H|_Z$ is not strictly singular, which implies the claim.

Let $N=\{n\in N_{t_0}: \ i_{n,t_0}\in I\}$. Apply Fact \ref{fact2}(b) to the sequence $(e_{i_{n,t_0}})_{n\in N}$ obtaining a suitable normalized sequence $(w_k)$ with elements of the form $w_k=\sum_{n\in J_k}a_ne_{i_{n,t_0}}$, $k\in \N$.  

Now notice that
$$
\norm[\sum_{n\in J_k}a_nz_{n}(i_{n,t_0})e_{i_{n,t_0}}]=\norm[\sum_{n\in J_k}a_n\frac{\gamma_{t_0}}{4^{t_0}}e_{i_{n,t_0}}]=\frac{\gamma_{t_0}}{4^{t_0}}\norm[w_k]=\frac{\gamma_{t_0}}{4^{t_0}}
$$
whereas
\begin{align*}
\norm[\sum_{n\in J_k}a_n(z_{n}-z_{n}(i_{n,t_0})e_{i_{n,t_0}})]&=\norm[\sum_{t\in\N,t\neq t_0}\sum_{n\in J_k}a_nz_{n}(i_{n,t})e_{i_{n,t}}]\\
&\leq \sum_{t\in\N,t\neq t_0}\frac{\gamma_t}{4^{t}}\norm[\sum_{n\in J_k}a_ne_{i_{n,t}}]\\
&\leq \sup_{t\in\N}\norm[\sum_{n\in J_k}a_ne_{i_{n,t}}].
\end{align*}
As the vectors $\sum_{n\in J_k}a_ne_{i_{n,t}}$ have disjoint support with $w_k$ for any $k\in\N$, the last term converges to zero as $k\to\infty$ it follows that the projection $P_H|_Z$ is not strictly singular.

The next claim seems to be a rather natural requirement. 
\begin{claim}\label{claim4}
Take $G\subset\N$ with each of the sets $\{i_{n,t}: \ n\in N_t\}\setminus G$, $t\in\N$, finite. Then the projection $P_G|_Y$ is not strictly singular.  
\end{claim}
Proof of Claim \ref{claim4}. Notice that by (j3) for each $j\in\N$ we have $\supp v_{n,j}\subset G$ for all but finitely many $n$'s. Let $G'=\cup_{j\in 2\N}\cup_{n\in\N}\supp v_{n,j}\cap G$. We shall prove that the projection $P_{G'}|_Y$ is not strictly singular, which implies the claim. 

Recall that for $j\gg p$ then $\norm[v_{n_1,j}+\dots+v_{n_p,j}]\approx p/2$ \cite{KL}. Therefore by the assumption on $G$ for any $s,r\in\N$ and $\e>0$ we can pick $L\subset\N$ with $\# L=s$ and $j\in\N$ so that $\frac{1}{\# L} \sum_{n\in L} v_{n,2j}$ and $\frac{1}{\# L} \sum_{n\in L} v_{n,2j+1}$ are seminormalized $\ell_1^s$-averages with constant 2, with $\supp v_{n,2j}\subset G$ and $\supp v_{n,2j}>r$ for any $n\in L$. By construction of $(y_n)$ (precisely since $\norm[v_{n,j}]\geq 1/2$) it follows that also $\frac{1}{\# L}\sum_{n\in L}(y_n|_{G'})$ and $\frac{1}{\# L}\sum_{n\in L}(y_n|_{\N\setminus G'})$ are seminormalized $\ell_1^s$-averages with constant 4. It follows that we can pick a successive sequence $(L_s)$ such that the sequences $(u_s)$ and $(v_s)$, where
$$
u_s=\frac{1}{\# L_s}\sum_{n\in L_s}(y_n|_{G'}), \  v_s=\frac{1}{\# L_s}\sum_{n\in L_s}(y_n|_{\N\setminus G'}), \ s\in\N
$$ 
are seminormalized $\ell_1^s$-averages with constant 4, for any $s\in\N$. 

Now apply Fact \ref{fact2}(a) to the sequences $(u_s)$ and $(v_s)$, obtaining a normalized sequence 
$$
(\sum_{s\in J_k}a_s\frac{1}{\# L_s}\sum_{n\in L_s}(y_n|_{G'}))_{k\in \N},
$$ such that 
$$
\sum_{s\in J_k}a_s\frac{1}{\# L_s}\sum_{n\in L_s}(y_n|_{\N\setminus G'})\to 0, \ \  \ k\to\infty.
$$
This shows that the projection $P_{G'}|_Y$ is not strictly singular and ends the proof of the claim.
                                                       
In order to prove (D3) take a partition $F\cup G=\N$ and assume that $P_F|Z$ is strictly singular. By Claim \ref{claim3} for any $t\in\N$ the set $\{i_{n,t}: \ n\in N_t\}\cap F=\{i_{n,t}: \ n\in N_t\}\setminus G$ is finite, whereas by Claim \ref{claim4} the projection $P_G|Y$ is not strictly singular. This ends the proof of (D3) and thus of Theorem \ref{gow2}. 
\end{proof}
A natural question arises if one can find an operator which is not a strictly singular perturbation of a restriction of a diagonal operator inside any infinitely dimensional subspace of $X_U$. We shall discuss the proof of the above construction in any block subspace, with infinite Rapidly Increasing Sequences of special type playing the role of the basis of $X_U$ in the previous reasoning. 


The construction of an operator not of the form $D|_W+S$ in the space $X_U$ was based on the existence of a sequence generating an $\ell_{1}$-spreading model. As we have wrote above the existence of such sequence in Schlumprecht space was shown in \cite{KL} and was based on the finite representability  of $c_{0}$  in Schlumprecht space.  The finite representability  of $c_{0}$ in every block subspace of Schlumprecht space was later proved in \cite{M} and recently a new proof of this property concerning a variant of Gowers-Maurey space was given in \cite{FS}. Moreover, the authors show that the $c_{0}^{n}$ can be reproduced on a block sequence of a special type which also generates the basis of Schlumprecht space as a spreading model \cite{FS} Prop. 5.1. These  block sequences of special type, which can be found in any block subspace, were called a Special RIS (SRIS) according to their structure (\cite{FS} Def. 4.4). The proof rewritten in the case of Gowers unconditional space yields the first part of the following fact (let us notice here that the technical modification of the definition the original Gowers-Maurey space required for the main result of [7] are not needed for proving that SRIS generates the basis of Schlumprecht space).
\begin{fact}\label{fact6}
In every block subspace of $X_U$ there is a seminormalized SRIS $(x_i)$ such that 
\begin{enumerate}
 \item $(x_i)$ generates the basis $(\tilde{e}_i)$ of Schlumprecht space as a spreading model,
\item the mapping $\bar{e}_i\to x_i$, $i\in\N$, where $(\bar{e}_i) $ is the canonical basis of a variant of Schlumprecht space defined with use of the function $\sqrt{f}$ instead of $f$, extends to a bounded operator on a Schlumprecht space.
\end{enumerate}
 \end{fact}
The proof of the second part of the above fact follows the lines of the proof of Prop 5.8 in \cite{M}. In the sequel we shall use also the following simple observation: for a SRIS $(x_i)$ any finite sequence $(x_{k_1}, \dots, x_{k_N})$ with $k_1>N$ forms a RIS of $\ell_1$-averages.

\

Taking any sequence $(x_i)$ as in Fact \ref{fact6} we can again transfer the sequence $(\tilde{y}_n)$ of \cite{KL} generating an $\ell_1$-spreading model from Schlumprecht space to $[x_i]$ by substituting the basis $(\tilde{e}_i)$ with $(x_i)$ and repeat the construction of $(z_n)$. Recall that $\tilde{y}_n=\sum_{j=1}^n\frac{f(j)}{2j}\sum_{i\in I_j}\tilde{e}_i$, $\# I_j=j$, for each $n$, and take a sequence $(y_n)\subset [x_i]$ defined as
$$
y_n=\sum_{j=1}^nv_{n,j}=\sum_{j=1}^n\frac{f(j)}{2j}\sum_{i\in K_j}x_i, \ \ n\in\N
$$
again with $K_j$'s with $\# K_j=j$ pushed forward along the sequence $(x_i)$ to guarantee that the vectors $(y_n)$ form a seminormalized block sequence generating an $\ell_1$-spreading model. 

Repeat the definition the function $\tau:\cup_n\supp_{[x_i]} y_n\to \N$, taking into account the supports of $(y_n)$ with respect to the basic sequence $(x_i)$ instead of $(e_i)$. Set
$$
z_{n}=\sum_{j=1}^n\frac{f(j)}{2j}\sum_{i\in K_j}\frac{1}{4^{\tau(i)}}x_i, \ \ n\in\N
$$
and let $Y=[y_{n}]$, $Z=[z_{n}]$.

\

In order to repeat the proof of Theorem \ref{gow2} we shall need the following observation, which is a more precise formulation of Fact \ref{fact2}(a).
\begin{fact}\label{fact7} 
Fix a seminormalized block sequence $(u_n)\subset X_U$. Then for any  $c\geq 1$ and $\e>0$ there is a normalized vector $w=\sum_{n\in J}a_nu_n$, such that $\norm[\sum_{n\in J}a_nv_n]<\e$, for any RIS of $\ell_1$-averages $(v_1,\dots,v_{\#J})$ with constant $c$ and with $\supp u_n\cap \supp v_m=\emptyset$ for any $n,m\in\N$. 
\end{fact}
The proof of the property (D2) for $Y,Z$ can be rewritten in our case since $(x_i)$ generates $(\tilde{e}_i)$ as a spreading model by Fact \ref{fact6}(a), $(y_n)$ generates $\ell_1$-spreading model by \cite{KL} and the basis of a suitable variant of Schlumprecht space dominates $(x_i)$ by Fact \ref{fact6}(b). 

The proof of the property (D3) requires more attention since possible projections can split also the supports of $(x_i)$. However, a small modifications allow to repeat the reasoning. We repeat the notation of $i_{n,t}$ for any $n,t\in\N$, and $N_t$ for any $t\in\N$. Again for some  scalars $(\gamma_t)_{t\in\N}\subset [0,1]$ we have 
$$
y_n=\sum_{t\in \tau(\supp_{[x_i]}y_n)}\gamma_tx_{i_{n,t}}, \ \ z_n=\sum_{t\in \tau(\supp_{[x_i]}y_n)}\frac{\gamma_t}{4^t}x_{i_{n,t}}, \ \ n\in\N
$$ 
Then we have the following version of Claim \ref{claim3}.
\begin{claim}\label{claim5}
Take $F\subset\N$ with $P_{F}|_{Z}$ strictly singular.  Then for every $\e>0$ and $t\in\N$ the set $\{i_{n,t}: n\in N_t, \norm[P_{F}x_{i_{n,t}}]\geq\e\}$ is finite.
\end{claim}
Proof of Claim \ref{claim5}. On the contrary, assume that $\norm[P_{F}x_{i_{n,t_0}}]\geq \e$ for some $\e$, $t\in\N$ and infinitely many $n$'s. The collection of $i_{n,t_0}$'s denote by $H$. We shall prove that the mapping $P_J|_Z$ is not strictly singular, where $J=\cup_{i\in H}\supp x_{i_{n,t_0}}\cap F$, which will finish the proof. 

Assume first that $((I-P_F)x_{i_{n,t_0}})_n$ is seminormalized. Then by Fact \ref{fact5} $((I-P_{F})x_{i_{n,t_0}})_n$ is a sequence of $\ell_1$-averages of increasing length. Pick an infinite $M\subset\N$ so that any $N$ elements of the sequence $((I-P_{F})x_{i_{n,t_0}})_{n\in M}$ starting after $N$-th element form a RIS of $\ell_1$-averages. Now by Fact \ref{fact7} for any $k\in\N$ choose a vector $w_k=\sum_{n\in J_k}a_n P_Fx_{i_{n,t_0}}$, such that $\norm[\sum_{n\in J_k}a_nx_{i_{n,t}}]<1/2^k$, for any $t\in\N$, $t\neq t_0$ and $\norm[\sum_{n\in J_k}a_n(I-P_{F})x_{i_{n,t_0}}]\leq 1/2^k$. It follows that
$$
\norm[P_J\sum_{n\in J_k}a_nz_{n}]=\norm[\sum_{n\in J_k}a_n\frac{\gamma_{t_0}}{4^{t_0}}x_{i_{n,t_0}}]=\frac{\gamma_{t_0}}{4^{t_0}}\norm[w_k]=\frac{\gamma_{t_0}}{4^{t_0}}
$$
whereas
\begin{align*}
\norm[(I-P_J)\sum_{n\in J_k}a_nz_{n}]&=\norm[\sum_{t\in\N,t\neq t_0}\sum_{n\in J_k}a_n\frac{\gamma_t}{4^t}x_{i_{n,t}}+\sum_{n\in J_k}(I-P_F)\frac{\gamma_{t_0}}{4^t}x_{i_{n,t_0}}]\\
&\leq \sum_{t\in\N,t\neq t_0}\frac{1}{4^{t}}\norm[\sum_{n\in J_k}a_n\gamma_tx_{i_{n,t}}]+\frac{1}{4^{t_0}}\norm[\sum_{n\in J_k}a_n(I-P_F)\gamma_{t_0}x_{i_{n,t_0}}]\leq \frac{1}{2^k}.
\end{align*}
If $\liminf_n\norm[(I-P_F)x_{i_{n,t_0}}]=0$, then passing to a subsequence we can assume that $\norm[(I-P_F)x_{i_{n,t_0}}]\leq 1/2^n$, and in the above estimate we have control over $\norm[\sum_{n\in J_k}a_n(I-P_{F})x_{i_{n,t_0}}]$ straightforward. Therefore in both cases the above estimates prove that the projection $P_J|_Z$ is not strictly singular which yields contradiction.

On the other hand we have the following version of Claim \ref{claim4}.
\begin{claim}\label{claim6}
Take $G\subset\N$ with $(I-P_G)x_{i_{n,t}} \xrightarrow[n\to\infty]{}0$ for any $t\in \N$. Then $P_{G}|_{Y}$ is not strictly singular. 
\end{claim}
Proof of Claim \ref{claim6}.
For $G$ as in the claim by definition of $(y_n)$ and $\tau$ we have $(I-P_{G})v_{n,j}\xrightarrow[n\to\infty]{}0$ for every $j$. Now we repeat the reasoning from the proof  of Claim \ref{claim4} defining $G'$ in the same way and choosing successive $L_s\subset\N$ in such a way that the vectors
$$
w_s=\frac{1}{\# L_s}\sum_{n\in L_s}v_{n,2j}, \  x_s=\frac{1}{\# L_s}\sum_{n\in L_s}v_{n,2j+1}, \ s\in\N
$$
are seminormalized $\ell_1^s$-averages with constant 4 and with additional requirement that $\norm[(I-P_G)v_{n,2j}]<1/2^n$. The last condition guarantees that $(w_s)_s$ and $(P_{G'}w_s)_s$ are equivalent which allows for repeating the rest of the proof of Claim \ref{claim4}. 

Now in order to obtain the property (D3) for $Y$ and $Z$ take any partition $F\cup G=\N$ and assume that $P_F|_Z$ is strictly singular. Then by Claim \ref{claim5} for any $t\in\N$ we have $P_F(x_{i_{n,t}})\xrightarrow[n\to\infty]{}0$ which by Claim \ref{claim6} implies that $P_G|_Y$ is not strictly singular. Thus we proved that subspaces $Y=[y_n]$ and $Z=[z_n]$ satisfy (D1), (D2) and (D3). As by Fact \ref{fact6} such block subspaces can be found in any block subspace of $X_U$, by Lemma \ref{fact1} and a standard perturbation argument we get the following
\begin{theorem}\label{gow3}
For any infinitely dimensional subspace $X$ of Gowers unconditional space $X_U$ there is an operator defined on a subspace of $X$ which is not a strictly singular perturbation of a restriction of a diagonal operator on $X_U$. 
\end{theorem}
We close this section with an observation that the "diagonal + strictly singular" property holds for block subspaces of $X_U$. Namely we prove the following version of Prop. 7.5 and 7.6 of \cite{ABR} in case of Gowers unconditional space, generalizing Theorem 29 \cite{GM}. 
\begin{proposition}\label{gow1} Let $X_U$ be Gowers unconditional space, $Y$ - a block subspace of $X_U$. Then 

(i) any bounded diagonal-free operator $T: Y\to X_U$ is strictly singular,

(ii) any bounded operator $T: Y\to X_U$ is a strictly singular perturbation of a restriction of a diagonal operator on $X_U$. 
\end{proposition}
\begin{proof} By Prop. \ref{prop1}, as $X_U$ is tight by support, the second part follows from the first part. The proof of the first part is a variant of the proof of Prop. 7.5 of \cite{ABR} in our setting, which uses technique of \cite{ALT}, we include it for the sake of completeness. 

Take a bounded operator $T:Y\to X_U$, where $Y=[y_k]$ is a block subspace of $X_U$. Assume that $T$ is diagonal-free, i.e. $\supp Ty_k\cap\supp y_k=\emptyset$ for each $k\in\N$. We shall prove that for any sequence of $(x_n)$ of normalized $\ell_1^n$-averages $Tx_n$ converges to zero. By Fact \ref{fact4} it follows that $T$ is strictly singular, which ends the proof of the Proposition. 

Fix a block sequence $(x_n)\subset [y_k]$ of normalized $\ell_1^n$-averages. Passing to subsequence, after small perturbation, we can assume that $(Tx_{n}+x_{n})_{n\in N}$ is a block sequence. Write each $x_n$ as $x_{n}=\sum_{k\in A_{n}}a_{k}y_{k}$. For every  $B\subset \N$  denote by  $R_{B}$  the projection on $[e_{j}:j\in\cup_{i\in B}\supp y_{i}]$. 
\begin{claim}\label{partition}[(cf. Lemma 7.2 \cite{ABR})]
For any partitions $A_n=B_{n}\cup C_{n}$, $n\in\N$, we have  $\lim_{n} R_{C_{n}}TR_{B_{n}}x_{n}=0$.
\end{claim}
Proof of Claim \ref{partition}. Take partitions $A_n=B_n\cup C_n$, $n\in\N$, and assume that $\inf_{n\in N}\norm[R_{C_{n}}TR_{B_{n}}x_{n}]>0$ for some infinite $N\subset\N$. Then, as $T$ is bounded, $\inf_{n\in N}\norm[R_{B_n}x_n]>0$. By Fact \ref{fact5} the sequence $(R_{B_n}x_n)_{n\in N}$ is also a sequence of $\ell_1$-averages of increasing length with a common constant. Apply Fact \ref{fact2}(a) to the seminormalized block sequence $u_n=R_{C_{n}}TR_{B_{n}}x_{n}$ and  $v_n=R_{B_{n}}x_{n}$ , $n\in N$, obtaining sequences $(z_k)$ and $(w_k)$ with $z_k=\sum_{n\in J_k}b_nR_{B_n}x_n$, $w_k=\sum_{n\in J_k}b_nR_{C_n}TR_{B_n}x_n$, $\norm[w_k]=1$, $k\in\N$, and $z_k\to 0$, which contradicts the boundedness of $T$ and ends the proof of the claim.

Let now
$$
\mathcal{P}_{n}=
\begin{cases}
\{(B,C):B\cup C=A_{n}, B\cap C=\emptyset, \# B=\# C=\# A_{n}/2\} , \,\,\,\textrm{if $\# A_{n}$  is even},
\\
\{(B,C):B\cup C=A_{n}, B\cap C=\emptyset, |\# B-\# C|=1\} , \,\,\textrm{if $\# A_{n}$ is odd},
\end{cases}
$$
and set $L_{n}$ to be the integer part of $\# A_{n}/2$. 

\begin{claim}\label{countlem}[cf. Lemma 7.4 \cite{ABR}] $R_{A_{n}}Tx_{n}=\frac{\lambda_n}{\# \mathcal{P}_n} \sum_{(B,C)\in \mathcal{P}_n}R_{B}TR_{C}x_n$, where 
$$
\lambda_n=
\begin{cases}
\ \frac{2L_n(2L_n-1)}{L_n^2}, \text{ if } \# A_{n} \text{ is even},
\\
\ \frac{2(2L_n+1)}{L_{n}+1}, \text{  if } \# A_{n} \text{ is odd}.
\end{cases}
$$
\end{claim}
Proof of Claim \ref{countlem}. Notice first that 
\begin{align*}\label{gml1}
R_{A_{n}}Tx_n&=R_{A_{n}}\sum_ka_k(\sum_{j\notin\supp y_{k}}e_j^*(Ty_k)e_j)\quad\textrm{since}\, \supp y_{k}\cap\supp Ty_{k}=\emptyset\\
&=\sum_{i\in A_{n}}\sum_{j\in\supp y_{i}}(\sum_{k: k\ne i}a_{k}e_{j}^{*}(Ty_{k}))e_{j}.
\notag
\end{align*}
whereas for any partition $(B,C)$ of $A_n$ we have
\begin{align*}
R_{B}TR_{C}x_{n}=\sum_{i\in B}\sum_{j\in\supp y_{i}}(\sum_{k\in C}a_{k}e^{*}_{j}(Ty_{k}))e_{j}.
\end{align*}
Fix $i\in A_n$ and $j\in \supp y_i$. We shall prove that 
$$
\sum_{k: k\ne i}a_{k}e_{j}^{*}(Ty_{k})=\frac{\lambda_n}{\# \mathcal{P}_n} \sum_{(B,C)\in \mathcal{P}_n}e_j^*(R_{B}TR_{C}x_n).
$$
Indeed, by the definition of $R_B$,  if  $e_{j}^{*}(R_{B}TR_{C}x_n)\neq 0$ then $i\in B$. Thus for any $k\ne i$ there are as many terms $a_{k}e_{j}^{*}(Ty_{k})$ in the sum $\sum_{(B,C)\in\mathcal{P}_{n}}e_{j}^{*}(R_{B}TR_{C}x_{n})$  as is the cardinality of the set $\{(B,C)\in\mathcal{P}_{n}: i\in B, k\in C\}$. The latter is equal to $\frac{\# \mathcal{P}_n}{\lambda_n}$, which ends the proof of the claim. 

The following claim ends the proof of Prop. \ref{gow1}.
\begin{claim}\label{final}
 $\lim_nTx_n= 0$.
\end{claim}
Proof of Claim \ref{final}. Assume $\inf_{n\in N}\norm[Tx_n]>0$ for some infinite $N\subset\N$. Notice that $A_nTx_n\to 0$. Indeed, by Claim \ref{countlem}, $A_nTx_n=\frac{\lambda_n}{\# \mathcal{P}_n}\sum_{(B,C)\in \mathcal{P}_n}R_{B}TR_{C}x_n$ for some $0<\lambda_n\leq 4$. On the other hand, by Claim \ref{partition} we have 
\begin{align*}
 \lim_{n}\ (\sup\{\norm[R_{C}TR_{B}x_{n}]: (B,C)\,\textrm{ partition of}\, A_{n}\}) =0 .
\end{align*}
Hence, after small perturbation, we can assume that $\supp Tx_n\cap \supp x_n=\emptyset$, $n\in N$ with $N$ infinite. Apply Fact \ref{fact2}(a) to $u_n=Tx_n$ and $v_n=x_n$, $n\in N$, obtaining sequences $(z_k)$ and $(w_k)$ with $z_k=\sum_{n\in J_k}b_nx_n$, $w_k=\sum_{n\in J_k}b_nTx_n$, $\norm[w_k]=1$, $k\in\N$, and $z_k\to 0$, which contradicts boundedness of $T$.  

\end{proof}

\end{document}